\documentclass[10pt]{amsart}
\usepackage{graphicx}
\usepackage{amscd}
\usepackage{amsmath}
\usepackage{amsthm}
\usepackage{amsfonts}
\usepackage{amssymb}
\usepackage{mathrsfs}
\usepackage{enumerate}
\usepackage{amsrefs}
\usepackage{xcolor}
\usepackage[colorlinks, citecolor=blue, linkcolor=red, pdfstartview=FitB]{hyperref}
\usepackage[latin1]{inputenc}

\newcommand{\bB}{{\mathbb{B}}}
\newcommand{\bC}{{\mathbb{C}}}

\newcommand{\bM}{{\mathbb{M}}}
\newcommand{\bN}{{\mathbb{N}}}

\newcommand{\bR}{{\mathbb{R}}}
\newcommand{\bS}{{\mathbb{S}}}

  \newcommand{\D}{{\mathcal{D}}}
  \newcommand{\E}{{\mathcal{E}}}
  \newcommand{\F}{{\mathcal{F}}}
  
\renewcommand{\H}{{\mathcal{H}}}

  \newcommand{\K}{{\mathcal{K}}}  
\renewcommand{\L}{{\mathcal{L}}}

\renewcommand{\S}{{\mathcal{S}}}
  
  \newcommand{\U}{{\mathcal{U}}}
  \newcommand{\V}{{\mathcal{V}}}

\newcommand{\fA}{{\mathfrak{A}}}
\newcommand{\fB}{{\mathfrak{B}}}

\newcommand{\fH}{{\mathfrak{H}}}

\newcommand{\fJ}{{\mathfrak{J}}}

\newcommand{\fS}{{\mathfrak{S}}}
\newcommand{\fT}{{\mathfrak{T}}}

\newcommand{\rC}{\mathrm{C}}

\newcommand{\eps}{\varepsilon}
\renewcommand{\phi}{\varphi}

\newcommand{\upchi}{{\raise.35ex\hbox{$\chi$}}}

\newcommand{\ol}{\overline}

\newcommand{\id}{\operatorname{id}}

\newcommand{\tr}{\operatorname{tr}}

\newtheorem{lemma}{Lemma}[section]
\newtheorem{theorem}[lemma]{Theorem}
\newtheorem{proposition}[lemma]{Proposition}
\newtheorem{corollary}[lemma]{Corollary}
\theoremstyle{definition}

\newtheorem{example}[lemma]{Example}

\begin{document}
\author{Rapha\"el Clou\^atre}
\address{Department of Mathematics, University of Manitoba, Winnipeg, Manitoba, Canada R3T 2N2}
\email{raphael.clouatre@umanitoba.ca\vspace{-2ex}}
\thanks{The author was partially supported by an NSERC Discovery Grant}
\begin{abstract}
We study restriction and extension properties for states on $\rC^*$-algebras with an eye towards hyperrigidity of operator systems. We use these ideas to provide supporting evidence for Arveson's hyperrigidity conjecture. Prompted by various characterizations of hyperrigidity in terms of states, we examine unperforated pairs of self-adjoint subspaces in a $\rC^*$-algebra. The configuration of the subspaces forming an unperforated pair is in some sense compatible with the order structure of the ambient $\rC^*$-algebra. We prove that commuting pairs are unperforated, and obtain consequences for hyperrigidity.  Finally, by exploiting recent advances in the tensor theory of operator systems, we show how the weak expectation property can serve as a flexible relaxation of the notion of unperforated pairs.
\end{abstract}
\title[Unperforated pairs and hyperrigidity]{Unperforated pairs of operator spaces and hyperrigidity of operator systems}

\maketitle

\section{Introduction}\label{S:intro}

The study of uniform algebras (i.e. closed unital subalgebras of commutative $\rC^*$-algebras), combines concrete function theoretic ideas with abstract algebraic tools \cite{gamelin1969}. It is a classical topic that has proven to be useful in operator theory. Indeed, a prototypical instance of a uniform algebra is the disc algebra of continuous functions on the closed unit disc which are holomorphic on the interior. Through a basic norm inequality of von Neumann, one can bring the analytic properties of the disc algebra to bear on the theory of contractions on Hilbert space. The seminal work of Sz.-Nagy and Foias on operator theory aptly illustrates the depth of this interplay \cite{nagy2010}, and to this day the link is still being exploited. 

In light of this highly successful symbiosis between operator theory and function theory, it is natural to look for further analogies. One may wish to transplant the sophisticated machinery available for uniform algebras in the setting of general operator algebras. This ambitious vision was pioneered by Arveson, who instigated an influential line of inquiry in his landmark paper \cite{arveson1969}. Therein, he introduced the notion of boundary representations for an operator system $\fS$, and proposed that these should be the non-commutative analogue of the so-called Choquet boundary of a uniform algebra. Furthermore, he noticed that these boundary representations could be used to construct a non-commutative analogue of the Shilov boundary as well. Although Arveson himself was not able to fully realize this program at the time, via the hard work of many hands \cite{hamana1979},\cite{dritschel2005},\cite{arveson2008},\cite{davidson2015} the $\rC^*$-envelope of an operator system was constructed by analogy with the classical situation. Nowadays, this circle of ideas is regarded as the appropriate non-commutative version of the Shilov boundary of a uniform algebra and it has emerged as a ubiquitous invariant in non-commutative functional analysis \cite{hamana1999},\cite{blecher2001},\cite{FHL2016},\cite{CR2017}.

Arveson also recognized that the  non-commutative Choquet boundary was a rich source of intriguing questions. For instance, in \cite{arveson2011} he proposed a tantalizing connection with approximation theory by recasting a classical phenomenon in operator algebraic terms. The classical setting is a result of Korovkin \cite{korovkin1953}, which goes as follows. For each $n\in \bN$, let $\Phi_n:C[0,1]\to C[0,1]$ be a positive linear map and assume that
\[
\lim_{n\to \infty}\|\Phi_n(f)-f\|=0
\]
for every $f\in \{1,x,x^2\}$. Then, it must be the case that
\[
\lim_{n\to \infty}\|\Phi_n(f)-f\|=0
\]
for every $f\in C[0,1]$. In other words, the asymptotic behaviour of the sequence $(\Phi_n)_n$ on the $\rC^*$-algebra $C[0,1]$ is uniquely determined by the operator system $\fS$ spanned by $1,x,x^2$. This striking phenomenon was elucidated  by several researchers (see for instance \cite{altomare2010} for a recent survey), but the perspective most relevant for our purpose here was offered by \v{S}a\v{s}kin \cite{saskin1967}, who observed that the key property of $\fS$ is that its Choquet boundary coincides with $[0,1]$. 

A natural non-commutative analogue of Korovkin-type rigidity would be an operator system $\fS\subset B(\H)$ with the property that for any sequence of unital completely positive linear maps
\[
\Phi_n:\rC^*(\fS)\to \rC^*(\fS), \quad n\in \bN
\]
such that
\[
\lim_{n\to \infty}\|\Phi_n(s)-s\|=0, \quad s\in \fS
\]
we must have
\[
\lim_{n\to \infty}\|\Phi_n(a)-a\|=0, \quad a\in \rC^*(\fS).
\]
In fact, Arveson introduced even more non-commutativity in this picture, and defined the operator system $\fS$ to be \emph{hyperrigid} if for any injective $*$-representation 
\[
\pi:\rC^*(\fS)\to B(\H_\pi)
\]
and for any sequence of unital completely positive linear maps
\[
\Phi_n:B(\H_\pi)\to B(\H_\pi), \quad n\in \bN
\]
such that
\[
\lim_{n\to \infty}\|\Phi_n(\pi(s))-\pi(s)\|=0, \quad s\in \fS
\]
we must have
\[
\lim_{n\to \infty}\|\Phi_n(\pi(a))-\pi(a)\|=0, \quad a\in \rC^*(\fS).
\]
Note that even in the case where $\rC^*(\fS)$ is commutative, a priori this phenomenon is stronger than the one observed by Korovkin, as we allow the maps $\Phi_n$ to take values outside of $\rC^*(\fS)$. Nevertheless, in accordance with \v{S}a\v{s}kin's insightful observation, Arveson conjectured \cite{arveson2011} that hyperrigidity is equivalent to the non-commutative Choquet boundary of $\fS$ being as large as possible, in the sense that every irreducible $*$-representation of $\rC^*(\fS)$ should be a boundary representation for $\fS$. This is now known as \emph{Arveson's hyperrigidity conjecture} and it has garnered significant interest in recent years \cite{kennedy2015},\cite{kleski2014hyper},\cite{NPSV2016},\cite{CH2016}. Arveson himself showed in \cite{arveson2011} that the conjecture is valid whenever $\rC^*(\fS)$ has countable spectrum.  Recently, it was verified in \cite{DK2016}  in the case where $\rC^*(\fS)$ is commutative. 

The hyperrigidity conjecture is the main motivation behind our  work here. Technically speaking however, the paper is centred around extensions and restrictions of states on $\rC^*$-algebras, and these issues occupy us for the majority of the article. We feel this approach to hyperrigidity is very natural, but as far as we know it has not been carefully investigated beyond the early connection realized in \cite[Theorem 8.2]{arveson2008}. In the final section of the paper, we introduce what we call ``unperforated pairs" of subspaces in a $\rC^*$-algebra. As we show, they constitute a device that can be leveraged to gain information about states, and ultimately to detect hyperrigidity. They also highlight a novel angle of approach to the hyperrigidity conjecture.

We now describe the organization of the paper more precisely.
In Section \ref{S:prelim} we gather the necessary background material. In particular, we recall that hyperrigidity of an operator system $\fS$ is equivalent to the following unique extension property: for every unital $*$-representation $\pi:\fA\to B(\H)$ and every unital completely positive linear map $\Pi:\fA\to B(\H)$ which agrees with $\pi$ on $\fS$, we have $\pi=\Pi$.  In Section \ref{S:charhyper}, we explore the link between hyperrigidity and two properties of states, namely the unique extension property and the pure restriction property. The first main result of that section establishes these properties as a tool to detect hyperrigidity. We summarize our findings   (Theorem \ref{T:purerestriction}, Corollary \ref{C:equivinclusion} and Theorem \ref{T:purestateuep}) in the following. 

\begin{theorem}\label{T:mainpurerestriction}
Let $\fS$ be an operator system and let $\fA=\rC^*(\fS)$. Assume that every irreducible $*$-representation of $\fA$ is a boundary representation for $\fS$. Let $\pi:\fA\to B(\H)$ be a unital $*$-representation and let $\Pi:\fA\to B(\H)$ be a unital completely positive extension of $\pi|_\fS$. The following statements are equivalent.
\begin{enumerate}

\item[\rm{(i)}] We have $\pi=\Pi$.

\item[\rm{(ii)}] We have $\Pi(\fA)\subset \pi(\fA)$.

\item[\rm{(iii)}] Every pure state on $\rC^*(\Pi(\fA))$ restricts to a pure state on $\pi(\fA)$.

\item[\rm{(iv)}] There is a family of states on $\rC^*(\Pi(\fA))$ which separate \\$(\Pi-\pi)(\fA)$ and  restrict to pure states on $\pi(\fA)$.

\item[\rm{(v)}] Every pure state on $\rC^*(\Pi(\fA))$ has the unique extension property with respect to $\pi(\fA)$.

\item[\rm{(vi)}] There is a family of pure states on $\rC^*(\Pi(\fA))$ which separate $(\Pi-\pi)(\fA)$ and have the unique extension property with respect to $\pi(\fA)$.

%\item[\rm{(vii)}] There is a family of states on $\rC^*(\Pi(\fA))$ which separate $(\Pi-\pi)(\fA)$ and have the unique extension property with respect to $\pi(\fA)$.
\end{enumerate}
\end{theorem}

The other main result of Section \ref{S:charhyper} provides evidence supporting Arveson's conjecture (Theorem \ref{T:nosp}).

\begin{theorem}\label{T:mainnosp}
Let $\fS$ be an operator system and let $\fA=\rC^*(\fS)$. Assume that every irreducible $*$-representation of $\fA$ is a boundary representation for $\fS$. Let $\pi:\fA\to B(\H)$ be a unital $*$-representation and let $\Pi:\fA\to B(\H)$ be a unital completely positive extension of $\pi|_\fS$.  Then, the subspace $(\pi-\Pi)(\fA)$ contains no strictly positive element.
\end{theorem}

In Section \ref{S:uep}, we delve deeper into the unique extension property for states. Based on a general construction (Theorem \ref{T:famuep}), we exhibit natural examples where the unique extension property is satisfied by an abundance of states, which is relevant in view of Theorem \ref{T:mainpurerestriction}.

Finally, in Section \ref{S:unperfor} we introduce the notion of an unperforated pair. A pair $(\fS,\fT)$ of self-adjoint subspaces in a unital $\rC^*$-algebra is said to be \emph{unperforated} if whenever $a\in \fS$ and $b\in \fT$ are self-adjoint elements with $a\leq b$, we may find another self-adjoint element $b'\in \fT$ such that $\|b'\|\leq \|a\|$ and $a\leq b'\leq b$. This provides a mechanism to construct families of states with pure restrictions (Theorem \ref{T:unpstate}). The precise relation to hyperrigidity is illustrated in the following (Corollaries \ref{C:unphyper} and \ref{C:unphypercomm}).

\begin{theorem}\label{T:mainunphyper}
Let $\fS$ be a separable operator system and let $\fA=\rC^*(\fS)$. Assume that every irreducible $*$-representation of $\fA$ is a boundary representation for $\fS$. Let $\pi:\fA\to B(\H)$ be a unital $*$-representation and let $\Pi:\fA\to B(\H)$ be a unital completely positive extension of $\pi|_\fS$. Then,  the pair $((\Pi-\pi)(\fA), \pi(\fA))$ is unperforated if and only if $\Pi=\pi$. In particular, this is satisfied if $(\Pi-\pi)(\fA)$ commutes with $\pi(\fA)$.
\end{theorem}

Unperforated pairs appear to be elusive in the absence of some form of commutativity. Accordingly, we aim to find a meaningful relaxation of that notion. Based on recent advances in the tensor theory of operator systems and the so-called tight Riesz interpolation property, we propose that the weak expectation property is an appropriate relaxation. Our position is substantiated by the following result (Theorem \ref{T:weakunperexample}).

\begin{theorem}\label{T:mainunperexample}
Let $\fA$ be a unital $\rC^*$-algebra and let $\fB\subset \fA$ be a unital separable $\rC^*$-subalgebra with the weak expectation property. Let $a\in \fA$ be a self-adjoint element and let $\eps>0$. Then, there is a sequence $(\beta_n)_n$ of self-adjoint elements in $\fB$ with the following properties.
\begin{enumerate}

\item[\rm{(1)}] We have  $\|\beta_n\|\leq (1+\eps)\|a\|$ for every $n\in \bN$ and
\[
\limsup_{n\to\infty}\|\beta_n\|\leq \|a\|.
\]

\item[\rm{(2)}] We have
\[
\limsup_{n\to \infty} \psi(\beta_n)\leq \inf\{ \psi(b):b\in \fB, b\geq a\}
\]
and 
\[
\sup\{\psi(c): c\in \fB,  c\leq a\} \leq \liminf_{n\to \infty} \psi(\beta_n)
\]
for every state $\psi$ on $\fB$.
\end{enumerate}
\end{theorem}

As an application of the previous result, we refine Theorem \ref{T:purestateuep} in the presence of the weak expectation property (Corollary \ref{C:wephyper}).

\section{Preliminaries}\label{S:prelim}
\subsection{Operator systems and completely positive maps}\label{SS:OACP}
Let $B(\H)$ denote the $\rC^*$-algebra of bounded linear operators on some Hilbert space $\H$. An \emph{operator system} $\fS$ is a unital self-adjoint subspace of $B(\H)$. Due to work of Choi and Effros \cite{choi1977}, operator systems can be defined in a completely abstract fashion, but the previous ``concrete" definition will suffice for our present purposes. Likewise, we will always assume that $\rC^*$-algebras are concretely represented on a Hilbert space. For each positive integer $n$, we denote by $\bM_n(\fS)$ the complex vector space of $n\times n$ matrices with entries in $\fS$, and regard it as a unital self-adjoint subspace of $B(\H^{(n)})$. A linear map 
\[
\phi:\fS\to B(\H_\phi)
\]
induces a linear map 
\[
\phi^{(n)}:\bM_n(\fS)\to B(\H_\phi^{(n)})
\]
defined as
\[
\phi^{(n)}([s_{ij}]_{i,j})=[\phi(s_{ij})]_{i,j}
\]
for each $[s_{ij}]_{i,j}\in \bM_n(\fS).$ The map $\phi$ is said to be \emph{completely positive} if $\phi^{(n)}$ is positive for every positive integer $n$.

For most of the paper, we will be dealing with unital completely positive maps with one-dimensional range. Such a map $\psi:\fS\to \bC$ is called a \emph{state}. The set of states on $\fS$ is denote by $\S(\fS)$. It is a weak-$*$ closed convex subset of the closed unit ball of the dual space $\fS^*$, and so in particular it is compact in the weak-$*$ topology.

The structure of unital completely positive maps on $\rC^*$-algebras is elucidated by the Stinespring construction, a generalization of the classical Gelfand-Naimark-Segal (GNS) construction associated to a state. More precisely,
given a unital $\rC^*$-algebra $\fA$ and a unital completely positive map $\phi:\fA\to B(\H)$, there is a Hilbert space $\fH_\phi$, an isometry $V_\phi:\H\to \fH_\phi$ and a unital $*$-representation $\sigma_\phi:\fA\to B(\fH_\phi)$ satisfying
\[
\phi(a)=V_\phi^* \sigma_\phi(a)V_\phi, \quad a\in \fA
\]
and  $\fH_\phi=[\sigma_\phi(\fA)V_\phi\H].$ Here and throughout, given a subset $\V\subset \H$ we denote by $[\V]$ the smallest closed subspace of $\H$ containing $\V$.
The triple $(\sigma_\phi, \fH_\phi,V_\phi)$ is called the \emph{Stinespring representation} of $\phi$, and it is unique up to unitary equivalence.
The following fact is standard.

\begin{lemma}\label{L:stinespring}
Let $\fA$ be a unital $\rC^*$-algebra and let $\fB\subset \fA$ be a unital $\rC^*$-subalgebra. Let $\psi:\fA\to B(\H)$ be a unital completely positive map and let $\phi=\psi|_\fB$. Then, there is an isometry $W:\fH_\phi\to \fH_\psi$ such that $WV_\phi=V_\psi$ and
\[
\sigma_\psi(b) W=W\sigma_\phi(b), \quad b\in \fB.
\]
\end{lemma}
\begin{proof}
We first note that if $b_1,\ldots,b_n\in \fB$ and $\xi_1,\ldots,\xi_n\in \H$ then
\begin{align*}
\left\|\sum_{j=1}^n \sigma_\psi(b_j)V_\psi \xi_j\right\|^2&=\sum_{j,k=1}^n\langle V_\psi^* \sigma_\psi(b_k^*b_j)V_\psi\xi_j,\xi_k\rangle =\sum_{j,k=1}^n\langle \psi(b_k^*b_j)\xi_j,\xi_k\rangle\\
&=\sum_{j,k=1}^n\langle \phi(b_k^*b_j)\xi_j,\xi_k\rangle=\sum_{j,k=1}^n\langle V_\phi^* \sigma_\phi(b_k^*b_j)V_\phi\xi_j,\xi_k\rangle\\
&=\left\|\sum_{j=1}^n \sigma_\phi(b_j)V_\phi \xi_j\right\|^2.
\end{align*}
Using that $\fH_\phi=[\sigma_\phi(\fB)V_\phi\H]$, a routine argument shows that there is an isometry $W:\fH_\phi\to \fH_\psi$ such that
\[
W\left( \sum_{j=1}^n \sigma_\phi(b_j)V_\phi \xi_j\right) =\sum_{j=1}^n \sigma_\psi(b_j)V_\psi \xi_j
\]
for every $b_1,\ldots,b_n\in \fB$ and $\xi_1,\ldots,\xi_n\in \H$. It follows readily that $WV_\phi=V_\psi$ and
\[
W\sigma_\phi(b)=\sigma_\psi(b)W, \quad b\in \fB.
\]
\end{proof}

\subsection{Purity, extreme points and Choquet integral representation}\label{SS:purity}

Let $\fS$ be an operator system. A completely positive map $\psi:\fS\to B(\H)$ is said to be \emph{pure} if whenever $\phi:\fS\to B(\H)$ is a completely positive map with the property that $\psi-\phi$ is also completely positive, we must have that $\phi=t\psi$ for some $0\leq t\leq 1$. It is known that the pure unital completely positive maps on a $\rC^*$-algebra are precisely those for which the associated Stinespring representations are irreducible \cite[Corollary 1.4.3]{arveson1969}.  

We let $\S_p(\fS)$ denote the collection of pure states. It is a standard fact that a state is pure if and only if it is an extreme point of $\S(\fS)$ (see for instance \cite[Proposition 2.5.5]{Dixmier1977}, the proof of which is easily adapted to the setting of an operator system). A subtlety arises for unital completely positive maps with higher dimensional ranges: it follows from \cite[Example 2.3]{WW1999} and \cite{farenick2000} that  a matrix state $\psi:\fS\to B(\bC^n)$ is pure if and only if it is a so-called \emph{matrix} extreme point. 
%Matrix extreme points are extreme points in the usual sense \cite[Corollary 3.6]{WW1999}, but these notions differ in general \cite[Example 2.2]{WW1999} of compact convex subsets having no matrix extreme points, so that the notion of matrix extreme points is different from the usual notion of extreme points. 

The following tool will be important for us. It follows from \cite[Theorem 4.2]{bishop1959} (see also \cite[Chapter 3]{phelps2001}). Recall that if $\fS$ is separable, then the weak-$*$ topology on $\S(\fS)$ is compact and metrizable.

\begin{theorem}\label{T:choquet}
Let $\fS$ be a separable operator system and let $\psi$ be a state on $\fS$. Then, there is a regular Borel probability measure on $\S(\fS)$ concentrated on $\S_p(\fS)$ and with the property that
\[
\psi(s)=\int_{\S_p(\fS)}\omega(s)d\mu(\omega), \quad s\in \fS.
\]
\end{theorem}

\subsection{Unique extension property, boundary representations and hyperrigidity}\label{SS:uep}

One important property of completely positive maps on operator systems is that they satisfy a generalization of the Hahn-Banach extension theorem. Indeed, let $\fS\subset B(\H)$ be an operator system and let $\phi:\fS\to B(\H_\phi)$ be a completely positive map. Then, by Arveson's extension theorem \cite{arveson1969}, there is another completely positive map $\psi:B(\H)\to B(\H_\phi)$ with the property that $\psi|_{\fS}=\phi$. In particular, a completely positive map on $\fS$ always admits at least one completely positive extension to any operator system $\fT\subset B(\H)$ containing $\fS$.  We denote the set of such extensions by $\E(\phi,\fT)$. This notation will be used consistently throughout the paper.

In general, the set of extensions may contain more than one element, and this possibility is one of the main themes of the paper. The following fact quantifies the freedom in choosing an extension, and it follows from a verbatim adaptation of the proof of \cite[Proposition 6.2]{arveson2011}.

\begin{lemma}\label{L:infext}
Let $\fS\subset \fT$ be operator systems and let $\phi$ be a state on $\fS$. Then,
\[
\max_{\psi\in \E(\phi,\fT)} \psi(t)= \inf\{\phi(s): s\in \fS, s\geq t \}
\]
and
\[
\min_{\psi\in\E(\phi,\fT)} \psi(t)= \sup\{\phi(s): s\in \fS, s\leq t \}
\]
whenever $t\in \fT$ is self-adjoint.
\end{lemma}

Let $\fS\subset \fT$ be operator systems. We say that a completely positive map $\psi:\fT\to B(\H_\psi)$ has the \emph{unique extension property with respect to $\fS$} if the restriction $\psi|_{\fS}$ admits only one completely positive extension to $\fT$, namely $\psi$ itself. An irreducible $*$-representation $\pi:\rC^*(\fS)\to B(\H_\pi)$ is said to be a \emph{boundary representation for $\fS$} if it has the unique extension property with respect to $\fS$. 

We advise the reader to exercise some care: in other works (such as \cite{arveson2008}) the use of the terminology ``unique extension property" is reserved for $*$-representations on $\rC^*(\fS)$. Our definition is more lenient as we do not restrict our attention to $*$-representations and no further relation is assumed between $\fS$ and $\fT$ beyond mere containment. We will recall this discrepancy in terminology whenever there is any risk of confusion.

These notions can be used to reformulate the property of hyperrigidity considered in the introduction. The following is \cite[Theorem 2.1]{arveson2011}; therein some special attention is paid to separability conditions, but a quick look at the proof reveals that the next result holds with no cardinality assumptions.

\begin{theorem}\label{T:arvhyperequiv}
Let $\fS$ be an operator system. Then, $\fS$ is hyperrigid if and only if every unital $*$-representation of $\rC^*(\fS)$ has the unique extension property with respect to $\fS$.
\end{theorem}

The driving force behind our work is the following conjecture of Arveson \cite{arveson2011}, which claims that it is sufficient to focus on irreducible $*$-representations to detect hyperrigidity.
\vspace{3mm}

\emph{Arveson's hyperrigidity conjecture.} An operator system $\fS$ is hyperrigid if every irreducible $*$-representation of $\rC^*(\fS)$ is a boundary representation for $\fS$.

\vspace{3mm}

To be precise, we should point out that Arveson was more cautious and restricted the operator system in his conjecture to be separable. We explain why this conjecture is especially sensible in that case. We may think of an arbitrary $*$-representation as some kind of integral of a family of irreducible $*$-representations against some measure. Since the irreducible $*$-representations are all assumed to have the unique extension property with respect to $\fS$, the question then becomes whether this property is preserved by the integration procedure. This rough sketch can be made precise, and in fact this was the philosophy used by Arveson in \cite{arveson2008}. One of the main contributions therein \cite[Theorem 6.1]{arveson2008} establishes that if the result of the integration procedure has the unique extension property with respect to $\fS$, then the integrand must have it almost everywhere. Arveson's hyperrigidity conjecture essentially asserts the converse. Note that in the ``atomic" situation where the integral is in fact a direct sum, this converse does indeed hold \cite[Proposition 4.4]{arveson2011}.

Finally, we note that we choose not to make separability of our operator systems a blanket assumption, although such conditions will occasionally make an appearance for technical reasons throughout.

\section{Characterizing hyperrigidity via states}\label{S:charhyper}

In this section, we make partial progress towards verifying the hyperrigidity conjecture and provide several different characterizations of hyperrigidity using states. 

Before proceeding, we make an observation that will be used numerous times throughout. Let $\fS$ be an operator system and let $\fA=\rC^*(\fS)$. Let $\pi:\fA\to B(\H)$ be a unital $*$-representation and let $\Pi:\fA\to B(\H)$ be a unital completely positive extension of $\pi|_\fS$. Then, we have
\begin{equation}\label{Eq:inclusion}
\pi(\fA)=\rC^*(\pi(\fS))=\rC^*(\Pi(\fS))\subset \rC^*(\Pi(\fA)).
\end{equation}

The basic tool of this section is the following. 

\begin{lemma}\label{L:pure}
Let $\fS$ be an operator system and let $\fA=\rC^*(\fS)$. Assume that every irreducible $*$-representation of $\fA$ is a boundary representation for $\fS$. Let $\pi:\fA\to B(\H)$ be a unital $*$-representation and let $\Pi:\fA\to B(\H)$ be a unital completely positive extension of $\pi|_\fS$. Then, we have that $\psi\circ \Pi=\psi\circ\pi$
whenever $\psi$ is a unital completely positive map on $ \rC^*(\Pi(\fA))$ with the property that $\psi|_{\pi(\fA)}$ is pure.
\end{lemma}
\begin{proof}
Recall that $\pi(\fA)\subset \rC^*(\Pi(\fA))$ by (\ref{Eq:inclusion}). Let $\phi=\psi|_{\pi(\fA)}$ which is pure by assumption. Let $(\sigma_\psi,\fH_\psi, V_\psi)$ and $(\sigma_\phi,\fH_\phi,V_\phi)$ denote the Stinespring representations for $\psi$ and $\phi$ respectively. By Lemma \ref{L:stinespring}, we see that there is an isometry $W:\fH_\phi\to \fH_\psi$ with the property that $WV_\phi=V_\psi$ and
\[
W^*\sigma_\psi(\pi(a))W=\sigma_\phi(\pi(a))
\]
for every $a\in \fA$. 
Since $\pi$ and $\Pi$ agree on $\fS$, we see that the map
\[
a\mapsto W^*\sigma_\psi(\Pi(a))W, \quad  a\in \fA
\]
is a unital completely positive extension of $\sigma_\phi \circ \pi|_{\fS}$. Because $\phi$ is pure, we infer that $\sigma_\phi$ is irreducible. In particular, $\sigma_\phi\circ \pi$ is an irreducible $*$-representation of $\fA$, and thus is a boundary representation for $\fS$.  We conclude that 
\[
W^*\sigma_\psi(\Pi(a))W=\sigma_\phi(\pi(a))
\]
for every $a\in \fA$. Hence, using that $WV_\phi=V_\psi$ we obtain
\begin{align*}
\psi(\Pi(a))&=V^*_\psi \sigma_\psi(\Pi(a))V_\psi\\
&=V^*_\phi W^*\sigma_\psi(\Pi(a))WV_\phi\\
&=V_\phi^*\sigma_\phi(\pi(a))V_\phi\\
&=\phi(\pi(a))=\psi(\pi(a))
\end{align*}
for every $a\in \fA$, and therefore $\psi\circ \Pi=\psi\circ \pi$.
\end{proof}

Our next task is to reformulate Lemma \ref{L:pure} in a language that is conveniently applicable to our purposes in the paper. Let $\fA$ be a unital $\rC^*$-algebra and let $\fS\subset \fA$ be a self-adjoint subspace. Let $\F$ be a collection of states on $\fA$. We say that the states in $\F$ \emph{separate $\fS$} if for every non-zero self-adjoint element $s\in \fS$ we have that
\[
\sup_{\psi\in \F}|\psi(s)|>0.
\]

\begin{theorem}\label{T:purerestriction}
Let $\fS$ be an operator system and let $\fA=\rC^*(\fS)$. Assume that every irreducible $*$-representation of $\fA$ is a boundary representation for $\fS$. Let $\pi:\fA\to B(\H)$ be a unital $*$-representation and let $\Pi:\fA\to B(\H)$ be a unital completely positive extension of $\pi|_\fS$. The following statements are equivalent.
\begin{enumerate}

\item[\rm{(i)}] We have $\pi=\Pi$.

\item[\rm{(ii)}] Every pure state on $\rC^*(\Pi(\fA))$ restricts to a pure state on $\pi(\fA)$.

\item[\rm{(iii)}] There is a family of states on $\rC^*(\Pi(\fA))$ which separate \\$(\Pi-\pi)(\fA)$ and  restrict to pure states on $\pi(\fA)$.
\end{enumerate}
\end{theorem}
\begin{proof}
If $\pi=\Pi$, then $\rC^*(\Pi(\fA))=\pi(\fA)$ so that (i) implies (ii). It is trivial that (ii) implies (iii) since $(\Pi-\pi)(\fA)\subset \rC^*(\Pi(\fA))$ by (\ref{Eq:inclusion}). Finally, assume that there is a family $\F$ of states on $\rC^*(\Pi(\fA))$ which separate $(\Pi-\pi)(\fA)$ and  restrict to pure states on $\pi(\fA)$. To establish $\Pi=\pi$, it suffices to show that 
\[
\sup_{\psi\in \F}|\psi(\Pi(a)-\pi(a))|=0
\]
for every self-adjoint element $a\in \fA$. This follows from an application of Lemma \ref{L:pure}. We conclude that (iii) implies (i).
\end{proof}

In view of the previous statement, we note in passing that it is generally not true that if every state on a unital $\rC^*$-algebra $\fA$ restricts to be pure on a unital $\rC^*$-subalgebra $\fB$, then $\fB=\fA$. Indeed, simply consider the trivial case of $\fB=\bC I$.

We extract an easy consequence related to hyperrigidity.

\begin{corollary}\label{C:equivinclusion}
Let $\fS$ be an operator system and let $\fA=\rC^*(\fS)$. Assume that every irreducible $*$-representation of $\fA$ is a boundary representation for $\fS$. Let $\pi:\fA\to B(\H)$ be a unital $*$-representation and let $\Pi:\fA\to B(\H)$ be a unital completely positive extension of $\pi|_\fS$. Then, $\pi=\Pi$ if and only if $\Pi(\fA)\subset \pi(\fA)$.
\end{corollary}
\begin{proof}
Assume that $\Pi(\fA)\subset \pi(\fA)$. Then, we have 
\[
\pi(\fS)=\Pi(\fS)\subset \Pi(\fA)\subset \pi(\fA)
\]
which implies that 
\[
\pi(\fA)=\rC^*(\pi(\fS))=\rC^*(\Pi(\fA)).
\]
Thus, the pure states on $\pi(\fA)$ coincide with those on $\rC^*(\Pi(\fA))$, and Theorem \ref{T:purerestriction} implies that $\pi=\Pi$. The converse is trivial. 
\end{proof}

In light of Theorem \ref{T:purerestriction}, it behooves us to understand the states on a unital $\rC^*$-algebra $\fA$ which restrict to be pure on  a unital $\rC^*$-subalgebra $\fB$. Fixing a state $\psi$ on $\fA$ and allowing $\fB$ to vary (while still being non-trivial), it is sometimes possible to arrange for the restriction $\psi|_\fB$ to be pure as well; see \cite{hamhalter2002} and references therein. Typically however, one does not expect purity of the restriction, as easy examples show.

\begin{example}\label{E:notpurerestriction}
Let $\bM_2$ be the complex $2\times 2$ matrices and let $\{e_1,e_2\}$ be the canonical orthonormal basis of $\bC^2$. Choose non-zero complex numbers $\gamma_1,\gamma_2$ such that $|\gamma_1|^2+|\gamma_2|^2=1$ and put
\[
\xi=\gamma_1 e_1+\gamma_2 e_2 .
\]
Define a state $\omega$ on $\bM_2$ as
\[
\omega(a)=\langle a\xi,\xi\rangle, \quad a\in \bM_2.
\]
The GNS representation of $\omega$ is seen to be unitarily equivalent to the identity representation on $\bM_2$, which is irreducible. Thus, $\omega$ is pure. Note however that the restriction of $\omega$ to the commutative $\rC^*$-subalgebra $\bC\oplus \bC\subset \bM_2$ is not multiplicative since both $\gamma_1$ and $\gamma_2$ are non-zero, and therefore the restriction is not pure.
\qed
\end{example}

 Nevertheless, the insight provided by Theorem \ref{T:purerestriction} will guide us throughout the paper, and it already contains non-trivial information regarding the hyperrigidity conjecture as we proceed to show next. First, we need a technical tool.

\begin{lemma}\label{L:nosp}
Let $\fS$ be an operator system and let $\fA=\rC^*(\fS)$. Assume that every irreducible $*$-representation of $\fA$ is a boundary representation for $\fS$. Let $\pi:\fA\to B(\H)$ be a unital $*$-representation and let $\Pi:\fA\to B(\H)$ be a unital completely positive extension of $\pi|_\fS$. Fix a state $\phi$ on $\pi(\fA)$ and an element $a\in \fA$ such that $\pi(a)-\Pi(a)$ is self-adjoint. Then, we have that
\[
\sup\{ \phi(\pi(c)):c\in \fA, \pi(c)\leq \pi(a)-\Pi(a)\}\leq 0.
\]
\end{lemma}
\begin{proof}
By (\ref{Eq:inclusion}), we have $\pi(\fA)\subset \rC^*(\Pi(\fA))$. Put $x=\pi(a)-\Pi(a) \in \rC^*(\Pi(\fA))$. We infer from Lemma \ref{L:pure} that $\psi(x)=0$ for every state $\psi$ on $\rC^*(\Pi(\fA))$ such that $\psi|_{\pi(\fA)}$ is pure. In particular, if $c\in \fA$ satisfies $\pi(c)\leq x$ and $\omega$ is a pure state on $\pi(\fA)$, then we see that 
\[
\omega(\pi(c))=\psi(\pi(c))\leq \psi(x)=0
\]
for every state $\psi$ on $\rC^*(\Pi(\fA))$ such that $\psi|_{\pi(\fA)}=\omega$. By the Krein-Milman theorem, the state $\phi$ lies in the weak-$*$ closure of the convex hull of $\S_p(\pi(\fA))$, and thus
\[
\sup\{ \phi(\pi(c)):c\in \fA,\pi(c)\leq x\}\leq 0.
\]
 \end{proof}

Let $\fS$ be an operator system and let $\fA=\rC^*(\fS)$. We assume that every irreducible $*$-representation of $\fA$ is a boundary representation for $\fS$. Further, let $\pi:\fA\to B(\H)$ be a unital $*$-representation and let $\Pi:\fA\to B(\H)$ be a unital completely positive map which agrees with $\pi$ on $\fS$. If the hyperrigidity conjecture holds, then we would have $\pi=\Pi$. In other words, the self-adjoint subspace $(\pi-\Pi)(\fA)$ would be trivial. The next development, which is one of the main result of this section, establishes that this subspace cannot contain any strictly positive element, thus supporting Arveson's conjecture.

\begin{theorem}\label{T:nosp}
Let $\fS$ be an operator system and let $\fA=\rC^*(\fS)$. Assume that every irreducible $*$-representation of $\fA$ is a boundary representation for $\fS$. Let $\pi:\fA\to B(\H)$ be a unital $*$-representation and let $\Pi:\fA\to B(\H)$ be a unital completely positive extension of $\pi|_\fS$.  Then, the subspace $(\pi-\Pi)(\fA)$ contains no strictly positive element.
\end{theorem}
\begin{proof}
Let $a\in \fA$ and assume that the element $x=\pi(a)-\Pi(a)$ is strictly positive, so that $x\geq \delta I$ for some $\delta>0$. We infer that 
 \[
\sup\{ \phi(\pi(c)):c\in \fA,\pi(c)\leq x\}\geq \phi(\pi(\delta I))=\delta
\]
for every state $\phi$ on $\fA$, which contradicts Lemma \ref{L:nosp}.
\end{proof}

Until now, the underlying theme of this section has been the purity of restrictions of states to $\rC^*$-subalgebras. The dual process of \emph{extending} states from a $\rC^*$-algebra to a larger one is also relevant for hyperrigidity, and we explore this idea next. We start by clarifying the relation between the unique extension property for states and the corresponding property for $*$-representations.

\begin{theorem}\label{T:srimpliesrr}
Let $\fA$ be a unital $\rC^*$-algebra and let $\fS\subset \fA$ be an operator system. The following statements hold.
\begin{enumerate}

\item[\rm{(1)}] Assume that every pure state on $\fA$ has the unique extension property with respect to $\fS$. Then, every irreducible $*$-representa\-tion of $\fA$ has the unique extension property with respect to $\fS$.

\item[\rm{(2)}] Assume that every state on $\fA$ has the unique extension property with respect to $\fS$. Then, every unital $*$-representation of $\fA$ has the unique extension property with respect to $\fS$. In particular, in the case where $\fA=\rC^*(\fS)$ we conclude that $\fS$ is hyperrigid.
\end{enumerate}
\end{theorem}
\begin{proof}
The two statements are established via near identical arguments, so we intertwine their proofs.
 Let $\pi:\fA\to B(\H)$ be a unital $*$-representation (irreducible in the case of (1)) and let $\Pi:\fA\to B(\H)$ be a unital completely positive map which agrees with $\pi$ on $\fS$. We must show that $\pi=\Pi$ on $\fA$, for which it is sufficient to establish that 
\[
\langle \pi(a) \zeta,\zeta\rangle=\langle \Pi(a) \zeta,\zeta\rangle
\]
for every unit vector $\zeta\in \H$ and every $a\in \fA$. Indeed, in that case, for each $a\in \fA$ the numerical radius of $\pi(a)-\Pi(a)\in B(\H)$ is $0$, and thus $\pi(a)=\Pi(a)$. 

Fix henceforth a unit vector $\zeta\in \H$ and consider the state $\chi$ on $\fA$ defined as
\[
\chi(a)=\langle \pi(a)\zeta,\zeta \rangle, \quad a\in \fA.
\]
If $\pi$ is irreducible, it is routine to verify that the GNS representation of $\chi$ is unitarily equivalent to $\pi$, whence $\chi$ must be pure in this case. Consider now another state $\psi$ on $\fA$ defined as
\[
\psi(a)=\langle \Pi(a)\zeta,\zeta \rangle, \quad a\in \fA.
\]
We see that $\psi$ and $\chi$ agree on $\fS$, whence they agree on $\fA$ by assumption. In other words,
\[
\langle \pi(a) \zeta,\zeta\rangle=\langle \Pi(a) \zeta,\zeta\rangle, \quad a\in \fA
\]
and the proof is complete.
\end{proof}

In the classical case where $\rC^*(\fS)$ is commutative, the pure states coincide with the irreducible $*$-representations, whence the converse of part (1) of Theorem \ref{T:srimpliesrr} holds. For general operator systems $\fS$ however, it can happen that $\fS$ is hyperrigid while there are some pure states on $\rC^*(\fS)$ which do not have the unique extension property with respect to $\fS$. We provide an elementary example.

\begin{example}\label{E:ueprepstates}
Let $\{e_1,e_2\}$ denote the canonical orthonormal basis of $\bC^2$. Consider the associated standard matrix units $E_{12},E_{21} \in\bM_2$. Let $\fS\subset \bM_2$ be the operator system generated by $I,E_{12},E_{21}$. Then, $\bM_2=\rC^*(\fS)$.  For $1\leq k \leq 2$, we let $\chi_k$ be the vector state on $\bM_2$ defined as
\[
\chi_k(a)=\langle ae_k, e_k\rangle, \quad a\in \bM_2.
\]
We see that the GNS representations of $\chi_1$ and $\chi_2$ are unitarily equivalent to the identity representation on $\bM_2$, which is irreducible. Thus, $\chi_1$ and $\chi_2$ are both pure. Moreover, every element $s\in \fS$ has the form
\[
s=c_0 I+c_{12}E_{12}+c_{21}E_{21}
\]
for some $c_0,c_{12},c_{21}\in \bC$. We note that
\[
\chi_k(c_0 I+c_{12}E_{12}+c_{21}E_{21})=c_0, \quad 1\leq k \leq 2
\]
so that $\chi_1|_{\fS}=\chi_2|_{\fS}$ while $\chi_1\neq \chi_2$. Thus, $\chi_1$ does not have the unique extension property with respect to $\fS$.

On the other hand, it is well-known that up to unitary equivalence the only unital $*$-representations of $\bM_2$ are multiples of the identity representation. The identity representation is a boundary representation for $\fS$ by Arveson's boundary theorem \cite[Theorem 2.1.1]{arveson1972}. Since the unique extension property is preserved under direct sums \cite[Proposition 4.4]{arveson2011}, we conclude that every unital $*$-representation of $\bM_2$ has the unique extension property with respect to $\fS$.
\qed
\end{example}

Our next task is to relate the unique extension property to the pure restriction property.  For this purpose, we recall some well-known facts which follow easily from standard convexity arguments (see Subsection \ref{SS:uep} about the notation used here).

\begin{lemma}\label{L:srchar}
Let $\fS\subset \fT$ be operator systems. 
\begin{enumerate}

\item[\rm{(1)}]  Let $\chi$ be a pure state on $\fT$ which has the unique extension property with respect to $\fS$. Then,  the restriction $\chi|_{\fS}$ is pure.

\item[\rm{(2)}]  Let $\omega$ be a pure state on $\fS$. Then, $\E(\omega,\fT)$ is a weak-$*$ closed convex subset of $\S(\fT)$ whose extreme points are pure states on $\fT$. In particular, $\omega$ admits a unique state extension to $\fT$ if and only if it admits a unique pure state extension to $\fT$.  

\end{enumerate}
\end{lemma}

In view of this interplay between the unique extension property for states and the pure restriction property, we give another characterization of hyperrigidity. 

%We first need the following fact.
%
%\begin{lemma}\label{L:uepseppure}
%Let $\fS$ be a separable operator system and let $\fT$ be another operator system which contains $\fS$. Let $\phi$ be a state on $\fS$ and let $t\in \fT$ be a self-adjoint element. Then, there is a pure state $\psi$ on $\fT$ which restricts to be pure on $\fS$ and such that
%\[
%\min_{\theta\in \E(\phi,\fT)}|\theta(t)|\leq |\psi(t)|.
%\]
%\end{lemma}
%\begin{proof}
%Since $\fS$ is separable, by virtue of Theorem \ref{T:choquet} there is a Borel probability measure $\mu$ concentrated on $\S_p(\fS)$ with the property that
%\[
%\phi(s)=\int_{\S_p(\fS)}\omega(s) d\mu, \quad s\in \fS.
%\]
%Next, for each pure state $\omega$ on $\fS$, by Lemma \ref{L:pureext} we may choose a pure state $\psi_\omega$ on $\fT$ which extends it. Define a state $\theta_0$ on $\fT$ as
%\[
%\theta_0(x)=\int_{\S_p(\fS)}\psi_\omega(x) d\mu(\omega), \quad x\in \fT
%\]
%and observe that $\theta_0$ agrees with $\phi$ on $\fS$. In particular, we have
%\begin{align*}
%\min_{\theta\in \E(\phi,\fT)}|\theta(t)|&\leq |\theta_0(t)|\leq \int_{\S_p(\fS)}|\psi_\omega(t)| d\mu
%\end{align*}
%whence there is a pure state $\omega$ on $\fS$ with the property that
%\[
%|\psi_\omega(t)|\geq \min_{\theta\in \E(\phi,\fT)}|\theta(t)|.
%\]
%The proof is complete.
%\end{proof}

\begin{theorem}\label{T:purestateuep}
Let $\fS$ be an operator system and let $\fA=\rC^*(\fS)$. Assume that every irreducible $*$-representation of $\fA$ is a boundary representation for $\fS$. Let $\pi:\fA\to B(\H)$ be a unital $*$-representation and let $\Pi:\fA\to B(\H)$ be a unital completely positive extension of $\pi|_\fS$. The following statements are equivalent.
\begin{enumerate}

\item[\rm{(i)}] We have $\pi=\Pi$.

\item[\rm{(ii)}] Every pure state on $\rC^*(\Pi(\fA))$ has the unique extension property with respect to $\pi(\fA)$.

\item[\rm{(iii)}] There is a family of pure states on $\rC^*(\Pi(\fA))$ which separate $(\Pi-\pi)(\fA)$ and have the unique extension property with respect to $\pi(\fA)$.
\end{enumerate}
%If $\fA$ is separable, then the previous statements are all equivalent to the following.
%\begin{enumerate}
%\item[\rm{(iv)}] There is a family of states on $\rC^*(\Pi(\fA))$ which separate \\
%$(\Pi-\pi)(\fA)$ and have the unique extension property with respect to $\pi(\fA)$.
%\end{enumerate}
\end{theorem}
\begin{proof}
If $\pi=\Pi$, then $\rC^*(\Pi(\fA))=\pi(\fA)$ so that (i) implies (ii). It is trivial that (ii) implies (iii) since $(\Pi-\pi)(\fA)\subset \rC^*(\Pi(\fA))$ by (\ref{Eq:inclusion}). Assume that there is a family of pure states on $\rC^*(\Pi(\fA))$ which separate $(\Pi-\pi)(\fA)$ and have the unique extension property with respect to $\pi(\fA)$. By Lemma \ref{L:srchar}, this family consists of states which restrict to be pure on $\pi(\fA)$. Thus, $\pi=\Pi$ by virtue of Theorem \ref{T:purerestriction}, and (iii) implies (i). 
%It is trivial that (iii) implies (iv) (regardless of separability). 
%Finally assume that $\fA$ is separable and that there is a family of states on $\rC^*(\Pi(\fA))$ which separate $(\Pi-\pi)(\fA)$ and have the unique extension property with respect to $\pi(\fA)$. Using Lemma \ref{L:uepseppure}, we see that there is there is a family of pure states on $\rC^*(\Pi(\fA))$ which separate $(\Pi-\pi)(\fA)$ and which restrict to be pure on $\pi(\fA)$. In light of Theorem \ref{T:purerestriction}, we see that (iv) implies (i).
\end{proof}

\section{The unique extension property for states}\label{S:uep}

In the previous section, we gave several different characterizations of hyperrigidity in terms of states. In particular, Theorem \ref{T:purestateuep} provides motivation to examine the unique extension property for states in greater detail. This is the task we undertake in this section. First, we remark that these uniqueness considerations for pure states  on a maximal abelian self-adjoint subalgebra of $B(\H)$ were at the heart of the famous Kadison-Singer problem \cite{KS1959} which was solved in \cite{MSS2015}. The case of general subalgebras has also been studied extensively; see \cite{archbold2001} and references therein.  

We now turn to examining the unique extension property for states which are not pure. In the separable setting, those are exactly the states which are given as the integration of a collection of pure states with respect to some probability measure (see Theorem \ref{T:choquet}). More generally, we have the following.

\begin{proposition}\label{P:uepint}
Let $\fA$ be a unital $\rC^*$-algebra and let $\fS\subset \fA$ be an operator system.  Let $\psi$ be a state on $\fA$ such that
\[
\psi=\int_{\S(\fA)}\chi d\mu(\chi)
\]
for some Borel probability measure $\mu$ on $\S(\fA)$. Assume that $\psi$ has the unique extension property with respect to $\fS$. If $\chi$ is a state on $\fA$ satisfying $\mu(\{\chi\})>0$, then $\chi$ has the unique extension property with respect to $\fS$.
\end{proposition}
\begin{proof}
Fix a state $\chi_0$ on $\fA$ with the property that $\mu(\{\chi_0\})>0$. Choose a state $\eps_0$ on $\fA$ which agrees with $\chi_0$ on $\fS$. Next, define a state $\phi$ on $\fA$ as
\[
\phi(a)=\mu(\{\chi_0\})\eps_{0}(a)+\int_{\S(\fA)\setminus \{\chi_0\}}\chi(a)d\mu(\chi), \quad a\in \fA.
\]
Then, we see that $\phi$ and $\psi$ agree on $\fS$, and thus by assumption we must have that $\phi=\psi$. Therefore, for each $a\in \fA$ we find that
\begin{align*}
\mu(\{\chi_0\})\eps_{0}(a)&=\phi(a)-\int_{\S(\fA)\setminus \{\chi_0\}}\chi(a)d\mu(\chi)\\
&=\psi(a)-\int_{\S(\fA)\setminus \{\chi_0\}}\chi(a)d\mu(\chi)\\
&=\mu(\{\chi_0\})\chi_0(a).
\end{align*}
Since $\mu(\{\chi_0\})>0$, we conclude that $\eps_0=\chi_0$ and the proof is complete.
\end{proof}

In particular, the previous proposition implies that if a state with the unique extension property is given as some finite convex combination of states, then all of those have the unique extension property as well. The question asking whether the condition on $\chi$ being an ``atom" of $\mu$ can be removed from the statement appears to be difficult. A related fact is known to hold at least in the case where $\fS$ is separable; see \cite[Theorem 6.1]{arveson2008}. In the setting of that paper however, states are replaced by $*$-representations and it is systematically assumed that $\fA=\rC^*(\fS)$. Under these conditions, the unique extension property is known to be equivalent to a dilation theoretic maximality property \cite[Proposition 2.4]{arveson2008}. This important characterization was first discovered in \cite{MS1998} and exploited with great success in \cite{dritschel2005}. It plays a crucial role in Arveson's proof of \cite[Theorems 5.6 and 6.1]{arveson2008}, and the lack of an analogue in our context is a major obstacle to adapting his ideas.

In the other direction, we exhibit an example which shows that integrating a collection of pure states with the unique extension property against some probability measure does not necessarily preserve the unique extension property.

\begin{example}\label{E:uepint}
Let $\bB_2\subset \bC^2$ denote the open unit ball and let $\bS_2$ be its topological boundary, the sphere. Let $A(\bB_2)$ denote the \emph{ball algebra}, that is the algebra of continuous functions on $\ol{\bB_2}$ which are holomorphic on $\bB_2$. Endow this algebra with the supremum norm over $\ol{\bB_2}$. By means of the maximum modulus principle, we may regard $A(\bB_2)$ as a unital closed subalgebra of $C(\bS_2)$. Let 
\[
\fS=A(\bB_2)+A(\bB_2)^*\subset C(\bS_2)
\]
be the operator system generated by $A(\bB_2)$ inside of $C(\bS_2)$. For every $\lambda\in \ol{\bB_2}$, denote by $\eps_\lambda$  the state on $\fS$ uniquely determined by 
\[
\eps_\lambda(f)=f(\lambda),\quad f\in A(\bB_2).
\]
It is a classical fact \cite{gamelin1969} that $\bS_2$ is the Choquet boundary of $A(\bB_2)$. Hence, for each $\zeta\in \bS_2$ the state $\eps_\zeta$ on $\fS$ has a unique extension to a state $\chi_\zeta$ on $C(\bS_2)$. This state $\chi_\zeta$ is in fact the character on $C(\bS_2)$ of evaluation at $\zeta$, and in particular it is pure.

Consider now the unique rotation invariant regular Borel probability measure  $\sigma$ on $\bS_2$, and let $\psi_\sigma$ denote the state on $C(\bS_2)$ of integration against $\sigma$. By virtue of Cauchy's formula \cite[Equation 3.2.4]{rudin2008}, we have that $\psi_\sigma(f)=f(0)$ for every $f\in A(\bB_2)$. On the other hand, let $\mu$ denote Lebesgue measure on the circle 
\[
\{(\zeta_1,\zeta_2)\in \bS_2: |\zeta_1|=1,\zeta_2=0\}
\]
 and let $\psi_\mu$ denote the state on $C(\bS_2)$ of integration against $\mu$. Then, $\psi_\mu\neq \psi_\sigma$. The one-variable version of Cauchy's formula shows that $\psi_\mu(f)=f(0)$ for every $f\in A(\bB_2)$. In particular, we conclude that $\psi_\sigma$ does \emph{not} have the unique extension property with respect to $\fS$. Finally, note that
\[
\psi_\sigma=\int_{ \bS_2}\chi_\zeta  d\sigma(\zeta)
\]
and we saw in the previous paragraph that each $\chi_\zeta$ is pure and has the unique extension property with respect to $\fS$. 
\qed
\end{example}

From the point of view of hyperrigidity, we see that Theorem \ref{T:purestateuep} offers some flexibility, in the sense that it only requires that there be sufficiently many states with the unique extension property. Accordingly, we next aim to identify a class of natural examples where the unique extension property is satisfied by a separating family of states. We start with a general result.

Recall that if $\fJ$ is a closed two-sided ideal of a $\rC^*$-algebra $\fA$, then $\fJ$ admits a \emph{contractive approximate identity}. In other words, there is a net $(e_\lambda)_{\lambda\in \Lambda}$ of positive elements $e_\lambda\in \fJ$ such that $\|e_\lambda\|\leq 1$ for every $\lambda\in \Lambda$ and with the property that
\[
\lim_\lambda \|be_\lambda-b\|=\lim_\lambda \|e_\lambda b-b\|=0
\]
for every $b\in \fJ$.

\begin{theorem}\label{T:famuep}
Let $\fA$ be a unital $\rC^*$-algebra and let $\fJ\subset \fA$ be a closed two-sided ideal with contractive approximate identity $(e_\lambda)_{\lambda\in \Lambda}$. Let $\chi$ be a state on $\fA$ such that 
\[
\lim_\lambda\chi(ae_\lambda)=\chi(a)
\]
for every $a\in \fA$. Then, $\chi$ has the unique extension property with respect to $\fJ+\bC I$.
\end{theorem}
\begin{proof}
Let $\psi$ be a state on $\fA$ which agrees with $\chi$ on $\fJ+\bC I$. Let $(\sigma_\psi,\fH_\psi,\xi_\psi)$ be the associated GNS representation, where $\xi_\psi\in \fH_\psi$ is a unit cyclic vector. 
Put $\fH_\fJ=[\sigma_\psi(\fJ)\fH_\psi]$, which is an invariant subspace for $\sigma_\psi(\fA)$.
We can decompose $\fH_\psi$ as
\[
\fH_\psi=\fH_\fJ\oplus \fH_\fJ'
\]
and accordingly we have
\[
\sigma_\psi(a)=\sigma_\psi(a)|_{\fH_\fJ}\oplus \sigma_\psi(a)|_{\fH_\fJ'}, \quad a\in \fA.
\]
If we let $\pi'_\fJ:\fA\to B(\fH_\fJ')$ be the unital $*$-representation defined as
\[
\pi'_{\fJ}(a)=\sigma_\psi(a)|_{\fH_\fJ'}, \quad a\in \fA
\]
then it is readily verified that $\pi'_{\fJ}(\fJ)=\{0\}.$ Hence, if we decompose
\[
\xi_\psi=\xi_\fJ\oplus \xi'_\fJ\in \fH_\fJ\oplus \fH_\fJ'
\]
then we observe that
 \[
 \chi(b)=\psi(b)=\langle \sigma_\psi(b)\xi_\psi,\xi_\psi \rangle=\langle\sigma_\psi(b)\xi_\fJ,\xi_\fJ \rangle
  \]
for every $b\in \fJ$. Note however that
\[
1=\chi(I)=\lim_\lambda \chi(e_\lambda)
\]
whence $\|\chi|_{\fJ}\|=1$. We conclude that $\|\xi_\fJ\|=1$ and $\xi'_\fJ=0$, whence $\xi_\psi=\xi_\fJ\in \fH_\fJ$. 
A standard verification then yields
\[
\lim_\lambda \sigma_\psi(e_\lambda)\xi_\psi=\xi_\psi
\]
in the norm topology of $\fH_\psi$.
On the other hand, we have that $ae_\lambda\in \fJ$ for each $a\in \fA$  and for each $\lambda\in \Lambda$, and thus
\begin{align*}
\chi(a)&=\lim_\lambda \chi(ae_\lambda)=\lim_\lambda \psi(ae_\lambda)=\lim_\lambda \langle \sigma_\psi(ae_\lambda)\xi_\psi,\xi_\psi\rangle\\
&=\lim_\lambda \langle \sigma_\psi(a)\sigma_\psi(e_\lambda)\xi_\psi,\xi_\psi\rangle=\langle \sigma_\psi(a)\xi_\psi,\xi_\psi\rangle\\
&=\psi(a).
\end{align*}
This completes the proof.
\end{proof}

We can now identify natural examples where many states have the unique extension property. Recall that a subset $\D\subset B(\H)$ is said to be \emph{non-degenerate} if $\ol{\D \H}=\H$.

\begin{corollary}\label{C:sepfamuep}
Let $\fA\subset B(\H)$ be a unital $\rC^*$-algebra and let $\fJ\subset \fA$ be a closed two-sided ideal which is non-degenerate. Let $X\in B(\H)$ be a positive trace class operator with $\tr X=1$, and let $\tau_X$ be the state on $\fA$ defined as
\[
\tau_X(a)=\tr(aX), \quad a\in \fA.
\]
Then, $\tau_X$ has the unique extension property with respect to $\fJ+\bC I$.
\end{corollary}
\begin{proof}
Let $(e_\lambda)_\lambda$ be a contractive approximate identity for $\fJ$. By assumption, we know that $\H=\ol{\fJ\H}$. A standard calculation then shows that $(e_\lambda)_{\lambda\in \Lambda}$ converges to the identity operator in the strong operator topology of $B(\H)$. Since $\tau_X$ is weak-$*$ continuous, we conclude that
\[
\lim_{\lambda}\tau_X(ae_\lambda)=\tau_X(a),\quad a\in \fA.
\]
By virtue of Theorem \ref{T:famuep}, we conclude that $\tau_X$ has the unique extension property with respect to $\fJ+\bC I$.
\end{proof}

Note that Theorem \ref{T:purestateuep} implies in particular that $\rC^*(\Pi(\fA))=\pi(\fA)$ if there is a family of pure states on $\rC^*(\Pi(\fA))$ which separate $(\Pi-\pi)(\fA)$ and have the unique extension property with respect to $\pi(\fA)$, assuming that every irreducible $*$-representation of $\fA$ is a boundary representation for $\fS$. We point out here that it is not  generally the case that two unital $\rC^*$-algebras $\fB\subset \fA$ coincide whenever there is a family of pure states on $\fA$ which separate $\fA$ and have the unique extension property with respect to $\fB$. The next example illustrates this phenomenon, along with the various properties of states considered thus far.

\begin{example}\label{E:purerestriction}
Let $\H$ be an infinite dimensional Hilbert space. Let $\fB$ be the $\rC^*$-algebra generated by the identity $I$ and the ideal of compact operators $\K(\H)$. Clearly, $\fB\neq B(\H)$.  
Recall that any non-degenerate $*$-representation of $\K(\H)$ is unitarily equivalent to some multiple of the identity representation. Standard facts about the representation theory of $\rC^*$-algebras (see the discussion preceding \cite[Theorem I.3.4]{arveson1976inv}) then imply that any unital $*$-representation of $B(\H)$ is unitarily equivalent to 
\[
\id^{(\gamma)}\oplus \pi_Q
\]
where $\gamma$ is some cardinal number and $\pi_Q$ is a $*$-representation of $B(\H)$ which annihilates $\K(\H)$. In light of the GNS construction, this shows that a pure state on $B(\H)$ is either a vector state or it annihilates $\K(\H)$. For a general state $\psi$ on $B(\H)$, we have the decomposition
\[
\psi=\tau+\psi_Q
\]
where $\tau$ is a positive weak-$*$ continuous linear functional on $B(\H)$, and $\psi_Q$ is a positive linear functional on $B(\H)$ which annihilates $\K(\H)$. Furthermore, there is a positive trace class operator $X_\psi\in B(\H)$ such that 
\[
\tau(a)=\tr(aX_\psi),\quad a\in B(\H).
\]
We now  carefully analyze the states on $B(\H)$ using this description.

First, note that if $\psi=\tau+\psi_Q$ where both $\tau$ and $\psi_Q$ are non-zero, then the restriction  $\rho=\psi|_{\fB}$ is not pure. For then $\rho-\tau|_{\fB}$ and $\rho-\psi_Q|_{\fB}$ are positive linear functionals. However, $\tau|_{\fB}$ and $\psi_Q|_{\fB}$ cannot be linearly dependent as they are both non-zero, and $\psi_Q$ annihilates $\K(\H)$ while $\tau$ does not. 

Second, assume that $\psi=\tau+\psi_Q$ where $\psi_Q$ is non-zero. We claim that $\psi$ does not have the unique extension property with respect to $\fB$. Indeed, since $B(\H)/\K(\H)$ is not merely one-dimensional and $\psi_Q(I)\neq 0$, there exists a positive linear functional $\chi$ on $B(\H)$ which annihilates $\K(\H)$ and satisfies $\chi(I)=\psi_Q(I)$ while $\chi\neq \psi_Q$. Then, the state $\tau+\chi$ agrees with $\psi$ on $\fB$, yet it is distinct from $\psi$.

Third, assume that $\psi=\psi_Q$. We claim that $\psi$ restricts to be pure on $\fB$. To see this, put $\rho=\psi|_\fB$ and suppose that there are states $\phi_1,\phi_2$ on $\fB$ with the property that
\[
\rho=\frac{1}{2}(\phi_1+\phi_2).
\]
Then, we have $\phi_1(K)=-\phi_2(K)$ for every $K\in \K(\H)$. Since $\phi_1$ and $\phi_2$ are positive, we conclude that
\[
\phi_1(K)=\phi_2(K)=0
\]
whenever $K\in \K(\H)$ is positive. Using the Schwarz inequality for states \cite[Proposition 3.3]{paulsen2002}, we see that
\[
|\phi_1(K)|^2\leq \phi_1(K^*K)=0,\quad K\in \K(\H).
\]
Hence $\phi_1$ annihilates $\K(\H)$, and so does $\phi_2$ by the same argument. Since $\phi_1(I)=\phi_2(I)=1$ we must have $\phi_1=\phi_2=\rho$. Therefore, $\rho$ is pure.

Finally, assume  that $\psi=\tau$. Then, $\psi$ has the unique extension property with respect to $\fB$ by virtue of Corollary \ref{C:sepfamuep}. Also, it is readily seen from  Lemma \ref{L:srchar} that  $\psi$ restricts to be pure on $\fB$ if and only if $X_\psi$ has rank one (i.e. $\psi$ is a vector state).

\qed
\end{example}

\section{Unperforated pairs of subspaces in a $\rC^*$-algebra}\label{S:unperfor}

In the previous section, we focused on the unique extension property for states, partly because it provides a means to produce a family of states on a $\rC^*$-algebra with the pure restriction property (see Theorem \ref{T:purestateuep} and its proof). In this section, we explore a different path and introduce a concept, which, under appropriate conditions, also leads to the identification of an abundance of states that restrict to be pure. 

Let $\fA$ be a unital $\rC^*$-algebra. Let $\fS$ and $\fT$ be self-adjoint subspaces of $\fA$. We say that the pair $(\fS,\fT)$ is \emph{unperforated} if for every pair of self-adjoint elements $a\in \fS,b\in \fT$ such that $a\leq b$, we can find another self-adjoint element $b'\in \fT$ with the property that $\|b'\|\leq \|a\|$ and $a\leq b'\leq b$. Clearly, the pair $(\fS,\fT)$ is automatically unperforated if $\fS\subset \fT$. 

We provide now an example of an unperforated pair $(\fS,\fT)$ for which there are self-adjoint elements $a\in \fS,b\in \fT$ with $a\leq b$ such that no element $b'\in \fT$ can be chosen to satisfy $a\leq b'\leq b$ and $\|b'\|=\|a\|$. 

\begin{example}\label{E:unpmatrices}
Let $\bM_3$ denote the $3\times 3$ complex matrices. Consider
\[
s=\begin{bmatrix}
-2 & 0  &0\\
0 & -1 & 0\\
0 & 0 & -1
\end{bmatrix}\in \bM_3
\]
and
\[
t=\begin{bmatrix}
1 & 0  &0\\
0 & -2 & 0\\
0 & 0 & 1
\end{bmatrix}\in \bM_3.
\]
Let $\fS=\bC s$ and $\fT=\bC t$, which are both self-adjoint subspaces of $\bM_3$. Let $a=\alpha s$ and $b=\beta t$ for some $\alpha\in \bC$ and $\beta\in \bC$. Assume that $a\leq b$, so that
\[
\begin{bmatrix}
-2\alpha & 0  &0\\
0 & -\alpha & 0\\
0 & 0 & -\alpha
\end{bmatrix}\leq \begin{bmatrix}
\beta & 0  &0\\
0 & -2 \beta & 0\\
0 & 0 & \beta
\end{bmatrix}.
\]
This is equivalent to the inequalities
\[
-2\alpha\leq \beta, \quad -\alpha\leq \beta\leq \alpha /2.
\]
In particular, we see that $\alpha\geq 0$ and $|\beta|\leq \alpha$. Thus, 
\[
\| b\|=2|\beta|\leq 2\alpha=\|a\|.
\]
We conclude that the pair $(\fS,\fT)$ is unperforated.  In fact, it has an additional noteworthy property. Choose $\alpha=1$ and $\beta=1/2$. Then, we trivially have that
\[
-2\alpha\leq \beta, \quad -\alpha\leq \beta\leq \alpha /2
\]
so the corresponding elements $a\in \fS$ and $b\in \fT$ satisfy $a\leq b$ as seen above. If $\lambda\in \bR$ satisfies $a\leq \lambda t\leq b$, then 
\[
\begin{bmatrix}
-2 & 0  &0\\
0 & -1 & 0\\
0 & 0 & -1
\end{bmatrix}\leq \begin{bmatrix}
\lambda & 0  &0\\
0 & -2 \lambda & 0\\
0 & 0 & \lambda
\end{bmatrix}\leq \begin{bmatrix}
1/2 & 0  &0\\
0 & -1 & 0\\
0 & 0 & 1/2
\end{bmatrix}
\]
which forces $\lambda=1/2$. We infer that $\|\lambda t\|=1<\|a\|$. 
\qed
\end{example}

We will give further  examples of unperforated pairs below (see Proposition \ref{P:unpcomm}).  In the meantime, we illustrate their usefulness for our purposes 
by leveraging their defining property.

\begin{theorem}\label{T:unpstate}
Let $\fA$ be a unital $\rC^*$-algebra, let $\fS\subset \fA$ be a self-adjoint subspace and let $\fT\subset \fA$ be a separable operator system. Assume that the pair $(\fS,\fT)$ is unperforated. Then, for every self-adjoint element $s\in \fS$, there is a state $\psi$ on $\fA$ which restrict to be pure on $\fT$ and such that $|\psi(s)|=\|s\|$. In particular, there is a family of states on $\fA$ which separate $\fS$ and restrict to be pure on $\fT$.
\end{theorem}
\begin{proof}
Fix a self-adjoint element $s\in \fS$. It is no loss of generality to assume that $\|s\|=1$. Upon replacing $s$ with $-s$, we can find a state $\theta$ on $\fA$ with the property that $\theta(s)=1$. Since $\fT$ is assumed to be separable we may invoke Theorem \ref{T:choquet} to find a Borel probability measure $\mu$ concentrated on $\S_p(\fT)$ with the property that 
\[
\theta(t)=\int_{\S_p(\fT)} \omega(t) d\mu(\omega), \quad t\in \fT.
\]
Assume on the contrary that for each pure state $\chi$ on $\fT$, we have that
\[
\max_{\psi\in \E(\chi,\fA)}|\psi(s)|<1.
\]
We will derive a contradiction by showing that $\mu(\S_p(\fT))=0$. To see this, first use Lemma \ref{L:infext}. We infer that for every pure state $\chi$ on $\fT$ we have that
\[
\inf\{ \chi(t):t\in \fT, t\geq s\}<1,
\]
and thus there is a self-adjoint element $t_\chi\in \fT$ such that $t_\chi\geq s$ and $\chi(t_\chi)<1$. Since the pair $(\fS, \fT)$ is unperforated, there is $t'_\chi\in \fT$ such that $\|t'_\chi\|\leq 1$ and $s\leq t'_\chi\leq t_\chi$. In particular, we note that $\chi(t'_\chi)<1$.
Consider now the weak-$*$ open set
\[
A_{\chi}=\{\omega\in \S_p(\fT): \omega(t'_{\chi})<1\}.
\]
Then, $\chi\in A_{\chi}$ and we see that
\[
\S_p(\fT)=\cup_{\chi\in \S_p(\fT)}A_{\chi}.
\]
Moreover, since $\|t'_{\chi}\|\leq 1$ and
\[
1=\theta(s)\leq \theta(t'_\chi)= \int_{\S_p(\fT)} \omega(t'_{\chi}) d\mu(\omega)
\]
we find $\mu(A_{\chi})=0$. 

By assumption, $\fT$ is separable and thus so is the subset
\[
Q=\{t'_\chi:\chi\in \S_p(\fT)\}.
\]
Accordingly let $\{\chi_n\}_{n\in \bN}$ be a countable subset of $\S_p(\fT)$ such that $\{t'_{\chi_n}\}_{n\in \bN}$ is dense in $Q$. Let $\chi\in \S_p(\fT)$ and $\omega\in A_{\chi}$, so that
\[
\omega(t'_{\chi})=1-\eps
\]
for some $\eps>0$. There is $N\in \bN$ such that $\|t'_{\chi_N}-t'_\chi\|<\eps/2$
whence 
\[
\omega(t'_{\chi_N})<\omega(t'_\chi)+\eps/2=1-\eps/2
\]
and $\omega\in A_{\chi_N}$. This shows that
\[
\S_p(\fT)=\cup_{\chi\in \S_p(\fT)}A_{\chi}= \cup_{n\in \bN}A_{\chi_n}.
\]
Since $\mu(A_{\chi_n})=0$ for every $n\in \bN$, we conclude that $\mu(\S_p(\fT))=0$. This contradicts the fact that $\mu$ has total mass $1$.
\end{proof}
Based on Theorem \ref{T:purerestriction}, we can now relate unperforated pairs and hyperrigidity. 
 
\begin{corollary}\label{C:unphyper}
Let $\fS$ be a separable operator system and let $\fA=\rC^*(\fS)$. Assume that every irreducible $*$-representation of $\fA$ is a boundary representation for $\fS$. Let $\pi:\fA\to B(\H)$ be a unital $*$-representa\-tion and let $\Pi:\fA\to B(\H)$ be a unital completely positive extension of $\pi|_\fS$. Then,  the pair $((\Pi-\pi)(\fA), \pi(\fA))$ is unperforated if and only if $\Pi=\pi$.
\end{corollary}
\begin{proof}
If $\Pi=\pi$, then $(\Pi-\pi)(\fA)=\{0\}\subset \pi(\fA)$ so that  the pair $((\Pi-\pi)(\fA), \pi(\fA))$ is trivially unperforated. 
Conversely, assume that the pair $((\Pi-\pi)(\fA), \pi(\fA))$ is unperforated. By Theorem \ref{T:unpstate}, there is a family of states on $\rC^*(\Pi(\fA))$ which separate $(\Pi-\pi)(\fA)$ and restrict to be pure on $\pi(\fA)$. Then $\Pi=\pi$ by virtue of Theorem \ref{T:purerestriction}.
\end{proof}

Next, we  exhibit a non-trivial condition which ensures that a pair $(\fS,\fT)$ is unperforated. 

\begin{proposition}\label{P:unpcomm}
 Let $\fA$ be a unital $\rC^*$-algebra. Let $\fS$ and $\fT$ be self-adjoint subspaces of $\fA$ such that $\fT$ is unital. Assume that $\fS$ commutes with $\fT$. Then, the pair $(\fS,\rC^*(\fT))$ is unperforated.
\end{proposition}
\begin{proof}
Let $a\in \fS,b\in \rC^*(\fT)$ be self-adjoint elements such that $a\leq b$. Define a continuous function $f:\bR\to \bR$ as 
\[
f(t)=\begin{cases}
t & \text{ if } |t|\leq \|a\|\\
\|a\| & \text{ if } t>\|a\|\\
-\|a\| & \text{ if } t<-\|a\|\\
\end{cases}.
\]
Observe that $f(a)=a$ and that $\|f(b)\|\leq \|a\|$ by choice of $f$. Now, since $a\leq b$ we must have $-\|a\| I\leq b$ and thus the spectrum of $b$ is contained in $[-\|a\|,\|b\|]$. Furthermore, we have that $f(t)\leq t$ for every $t\geq -\|a\|$. These two observations together show that $f(b)\leq b$. 

We claim that $a \leq f(b)$.  To see this, we note that $\rC^*(\fT)$ commutes with $\rC^*(\fS)$ since $\fS$ and $\fT$ are self-adjoint, whence the unital $\rC^*$-algebra $\rC^*(a,b,I)$ is commutative. Therefore, there is a compact Hausdorff space $\Omega$ and a unital $*$-isomorphism $\Phi:\rC^*(a,b,I)\to C(\Omega)$. Put $\phi_a=\Phi(a), \phi_b=\Phi(b)$. Recalling that $a=f(a)$, the claim is equivalent to the fact that $f\circ \phi_a\leq f\circ \phi_b$ on $\Omega$. 
Since we have that $a\leq b$, it follows that $\phi_a\leq \phi_b$. 
The function $f$ is non-decreasing, whence $f\circ \phi_a\leq f\circ \phi_b$ on $\Omega$ and the claim is established. Finally, the proof is completed by choosing $b'=f(b)$.
\end{proof}

In particular, we single out the following noteworthy consequence.

\begin{corollary}\label{C:unphypercomm}
Let $\fS$ be a separable operator system and let $\fA=\rC^*(\fS)$. Assume that every irreducible $*$-representation of $\fA$ is a boundary representation for $\fS$. Let $\pi:\fA\to B(\H)$ be a unital $*$-representation and let $\Pi:\fA\to B(\H)$ be a unital completely positive extension of $\pi|_\fS$. Assume that $(\Pi-\pi)(\fA)$ and $\pi(\fA)$ commute. Then, $\Pi=\pi$.
\end{corollary}
\begin{proof}
Simply combine Proposition \ref{P:unpcomm} with Corollary \ref{C:unphyper}.
\end{proof}

In trying to verify that a general pair $(\fS,\rC^*(\fT))$ is unperforated,  one may hope to proceed as in the proof of Proposition \ref{P:unpcomm} and use the functional calculus to ``truncate" $b$  inside of $\rC^*(\fT)$ to have norm at most $\|a\|$. However, in general it is not clear that this truncation should still dominate $a$. Indeed, the non-decreasing function $f$ defined in the proof is not operator monotone. In fact, there are many simple instances of non-unperforated pairs.

\begin{example}\label{E:perf}
Let $\fT=\bC\oplus \bC$ and let $\fS\subset \bM_2$ be the self-adjoint subspace generated by the matrix
\[
\begin{bmatrix}
0 & 1 \\
1 & 0
\end{bmatrix}.
\]
Then, the pair $(\fS,\fT)$ is not unperforated. Indeed, consider 
\[
a=\begin{bmatrix}
0 & 2 \\
2 & 0
\end{bmatrix}\in \fS, \quad b=\begin{bmatrix}
1 & 0 \\
0 & 5
\end{bmatrix}\in \fT
\]
and note that
\[
b-a=\begin{bmatrix}
1 & -2 \\
-2 & 5
\end{bmatrix}
\]
is positive, whence $a\leq b$. Let
\[
b'=\begin{bmatrix}
x & 0 \\
0 & y
\end{bmatrix}\in \fT
\]
be self-adjoint such that $a\leq b'$ and $\|b'\|\leq \|a\|=2$. 
Then,
\[
b'-a=\begin{bmatrix}
x & -2 \\
-2 & y
\end{bmatrix}\geq 0.
\]
In particular, we see that $x\geq 0, y\geq 0$ and $xy\geq 4$. Since $\|b'\|\leq 2$, we conclude that $\max\{x,y\}\leq 2$. Hence, $x=y=2$ so that
\[
b'=\begin{bmatrix}
2 & 0\\
0 & 2
\end{bmatrix}.
\]
But then
\[
b-b'=\begin{bmatrix}
-1 & 0\\
0 & 3
\end{bmatrix}
\]
is not positive.
\qed
\end{example}

In view of this difficulty, a pressing question emerges: how common are unperforated pairs? We saw in Proposition \ref{P:unpcomm} that they can be found easily in the presence of some form of commutativity, but Example \ref{E:perf} indicates the situation  may be bleak in general. Accordingly we aim to introduce flexibility in the defining condition for a pair to be unperforated. The key property we require is the following. 

A $\rC^*$-algebra $\fA$ is said to have the \emph{weak expectation property} \cite{lance1973} if for every injective $*$-representation $\pi:\fA\to B(\H_\pi)$, there is a unital completely positive map 
\[
E_\pi:B(\H_\pi)\to \pi(\fA)''
\]
satisfying $E\pi(a)=\pi(a)$ for every $a\in \fA$ (see for instance \cite{BO2008} for details). 
The next development shows that if $\fB\subset \fA$ are unital $\rC^*$-algebras, then the weak expectation property for $\fB$ may be viewed as a variation on the fact that the pair $(\fA,\fB)$ is unperforated. Interestingly, this fact uses (albeit indirectly) some recent technology from the theory of tensor products of operator systems.

\begin{theorem}\label{T:weakunperexample}
Let $\fA$ be a unital $\rC^*$-algebra and let $\fB\subset \fA$ be a unital separable $\rC^*$-subalgebra with the weak expectation property. Let $a\in \fA$ be a self-adjoint element and let $\eps>0$. Then, there is a sequence $(\beta_n)_n$ of self-adjoint elements in $\fB$ with the following properties.
\begin{enumerate}

\item[\rm{(1)}] We have  $\|\beta_n\|\leq (1+\eps)\|a\|$ for every $n\in \bN$ and
\[
\limsup_{n\to\infty}\|\beta_n\|\leq \|a\|.
\]

\item[\rm{(2)}] We have
\[
\limsup_{n\to \infty} \psi(\beta_n)\leq \inf\{ \psi(b):b\in \fB, b\geq a\}
\]
and 
\[
\sup\{\psi(c): c\in \fB,  c\leq a\} \leq \liminf_{n\to \infty} \psi(\beta_n)
\]
for every state $\psi$ on $\fB$.
\end{enumerate}
\end{theorem}
\begin{proof}
Assume that $\fB\subset \fA\subset B(\H)$. Consider the sets
\[
\U_a=\{b\in \fB:b\geq a\}, \quad \L_a=\{c\in \fB:c\leq a\}.
\] 
Since $\fB$ is separable, so are $\U_a$ and $\L_a$. Thus, there are countable dense subsets $\{u_n\}_{n\in \bN} \subset \U_a, \{\ell_n\}_{n\in \bN}\subset \L_a$. Because $\fB$ has the weak expectation property, it follows from \cite[Theorem 7.4]{kavruk2012} that it has the so-called tight Riesz interpolation property in $B(\H)$. Noting that $\fB$ is unital and that
\[
-(1+\eps n^{-1})\|a\|I< a<(1+\eps n^{-1})\|a\|I, \quad n\in \bN
\]
this interpolation property guarantees that for each $n\in \bN$ we can find a self-adjoint element $\beta_{n}\in \fB$ satisfying 
\[
-(1+\eps n^{-1})\|a\|I< \beta_{n}<(1+\eps n^{-1})\|a\|I
\]
and
\[
\ell_j- n^{-1}I<\beta_{n}<u_k+n^{-1} I
\]
for every $1\leq j,k\leq n$. In particular, we note that 
\[
\|\beta_{n}\| \leq (1+\eps)\|a\|, \quad n\in \bN
\]
and
\[
\limsup_{n\to\infty}\|\beta_n\|\leq \|a\|.
\]
Moreover it follows from the construction of the sequence $(\beta_n)_n$ that if $\psi$ is a state on $\fB$ then
\[
\sup_{m\in \bN}\psi(\ell_m)\leq \liminf_{n\to \infty}\psi(\beta_n)\leq  \limsup_{n\to \infty}\psi(\beta_n)\leq \inf_{m\in \bN} \psi(u_m).
\]
On the other hand, we have that
\[
 \inf_{m\in \bN} \psi(u_m)= \inf\{ \psi(b):b\in \fB, b\geq a\}
\]
and
\[
\sup_{m\in \bN}\psi(\ell_m)=\sup\{\psi(c): c\in \fB,  c\leq a\}
\]
by density. Hence,
\[
\limsup_{n\to \infty}\psi(\beta_n)\leq  \inf\{ \psi(b):b\in \fB, b\geq a\}
\]
and
\[
\sup\{\psi(c): c\in \fB,  c\leq a\} \leq \liminf_{n\to \infty} \psi(\beta_n).
\]
\end{proof}

Of course, the weak expectation property arises naturally without the need for any kind of commutativity, so that Properties (1) and (2) from Theorem \ref{T:weakunperexample} constitute a flexible substitute for the fact that the pair $(\fA,\fB)$ is unperforated. We substantiate this claim in what follows. We start with a concrete observation.

\begin{example}\label{E:weakunpexample}
Let $\H$ be an infinite dimensional separable Hilbert space. The unital separable $\rC^*$-algebra $\fB=\K(\H)+\bC I$ is nuclear since $\K(\H)$ is nuclear  \cite[Exercise 2.3.5]{BO2008}. In particular, it has the weak expectation property \cite[Exercise 2.3.14]{BO2008}. Next, let $\theta$ be a state on $B(\H)$ which has the unique extension property with respect to $\fB$. By Example \ref{E:purerestriction} we conclude that there is a positive trace class operator $X_\theta\in B(\H)$ with $\tr(X_\theta)=1$ and such that
\[
\theta(a)=\tr(aX_\theta), \quad a\in B(\H).
\]
Upon applying the spectral theorem to $X_\theta$, we may find a sequence of positive numbers  $(t_n)_{n\in\bN}$ and a sequence of orthonormal vectors $(\xi_n)_{n\in\bN}$ such that
\[
\theta(a)=\tr(aX_\theta)=\sum_{n=1}^\infty t_n \langle a\xi_n,\xi_n\rangle
\]
for every $a\in B(\H)$. In particular, we see that $\sum_{n=1}^\infty t_n=1$. Fix now a self-adjoint element $a\in B(\H)$. A moment's thought reveals that there must be $\xi\in \{\xi_n\}_{n\in \bN}$ with the property that
\[
|\langle a\xi,\xi \rangle|\geq |\theta(a)|.
\]
Furthermore, if we denote by $\chi$ the vector state on $B(\H)$ corresponding to $\xi$, we see from Example \ref{E:purerestriction} that $\chi$ restricts to be pure on $\fB$. This is a manifestation of a general phenomenon, as we show next.
%
%We claim that this is the case. To see this, fix a self-adjoint element $a\in \fA$ and choose a pure state $\chi$ on $\fA$ such that $|\chi(a)|=\|a\|$. Now, it is well-known that $\chi$ either is either a vector state or annihilates $\K(\H)$. If $\chi$ is a vector state, then $\chi$ is 
%
%
%Next, denote by $q:B(\H)\to B(\H)/\K(\H)$ the Calkin map. Assume that $a\in B(\H)$ is a self-adjoint element such that $\|q(a)\|<\|a\|$.
%
%We may thus apply Theorem \ref{T:weakunpnorm} to conclude that whenever $a\in B(\H)$ is a self-adjoint element with the property that
%\[
%\|a\|=\max_{\theta\in \P}|\theta(a)|
%\]
%then there is a pure state $\omega$ on $\fB$ such that
%\[
%\max_{\psi\in \E(\omega,\fA)}|\psi(s)|= \|s\|.
%\]
%
\qed
\end{example}
%
%We aim to show how our variation on unperforated pairs of subspaces (namely, Properties (1), (2) and (3) from Theorem \ref{T:weakunperexample}) can be used to construct separating families of states which restrict to pure states. 

\begin{theorem}\label{T:weppnorm}
Let $\fA$ be a unital $\rC^*$-algebra and let $\fB\subset \fA$ be a unital separable $\rC^*$-subalgebra with the weak expectation property.  Let $\theta$ be a state on $\fA$ which has the unique extension property with respect to $\fB$. Then, for every self-adjoint element $a\in \fA$ there is a state $\psi$ on $\fA$ which restricts to be pure on $\fB$ and such that
$|\psi(a)|\geq |\theta(a)|$.
\end{theorem}
\begin{proof}
Fix a self-adjoint element $a\in \fA$, which we may assume is non-zero without loss of generality. The desired conclusion is unchanged if we replace $a$ by $-a$, so we may assume that $\theta(a)\geq 0$. We argue by contradiction. Assume on the contrary that for each pure state $\omega$ on $\fB$ we have
\[
\max_{\psi\in \E(\omega,\fA)}|\psi(a)|<\theta(a).
\]
Then, we infer from Lemma \ref{L:infext} that 
\[
\inf\{ \omega(b):b\in \fB, b\geq a\}<\theta(a).
\]
Now, by Theorem \ref{T:weakunperexample} there is a sequence $(\beta_n)_{n\in \bN}$ of self-adjoint elements in $\fB$ with $\|\beta_n\|\leq 2\|a\|$ for every $n\in \bN$ and such that 
\[
\sup\{\theta(c): c\in \fB,  c\leq a\}\leq \liminf_{n\to \infty}\theta(\beta_n)
\]
and
\[
\limsup_{n\to \infty}\omega(\beta_n)\leq \inf\{\omega(b):b\in \fB, b\geq a\}<\theta(a)
\]
for every pure state $\omega$ on $\fB$. Since $\theta$ is assumed to have the unique extension property with respect to $\fB$, by Lemma \ref{L:infext} we find
\[
\theta(a)\leq \liminf_{n\to \infty}\theta(\beta_n).
\]
On the other hand, since $\fB$ is assumed to be separable we may invoke Theorem \ref{T:choquet} to find a Borel probability measure $\mu$ concentrated on $\S_p(\fB)$ with the property that 
\[
\theta(b)=\int_{\S_p(\fB)} \omega(b) d\mu(\omega), \quad b\in \fB.
\]
Upon applying Fatou's lemma to the sequence of non-negative continuous functions 
\[
\omega\mapsto 2\|a\|-\omega(\beta_n), \quad \omega\in \S_p(\fB)
\]
a simple calculation yields
\[
\limsup_{n\to \infty}\left(\int_{\S_p(\fB)}\omega(\beta_n) d\mu(\omega)\right)\leq \int_{\S_p(\fB)}\left(\limsup_{n\to \infty}\omega(\beta_n) \right)d\mu(\omega).
\]
Consequently
\begin{align*}
\limsup_{n\to \infty}\theta(\beta_n)&=\limsup_{n\to \infty}\left(\int_{\S_p(\fB)}\omega(\beta_n) d\mu(\omega)\right)\\
&\leq \int_{\S_p(\fB)}\left(\limsup_{n\to \infty}\omega(\beta_n) \right)d\mu(\omega)\\
&<\theta(a).
\end{align*}
But this implies that
\[
\theta(a)\leq  \liminf_{n\to \infty}\theta(\beta_n)\leq \limsup_{n\to \infty}\theta(\beta_n)<\theta(a)
\]
which is absurd.
\end{proof}

We mention a noteworthy consequence of Theorem \ref{T:weppnorm} which is related to hyperrigidity.

%: if there is a family $\F$ of states on $\fA$ which separate $\fA$ and all have the unique extension property with respect to $\fB$, then we can find a (potentially different) family $\F'$ of states on $\fA$ which separate $\fA$ and restrict to be pure on $\fB$. Of course, this is only meaningful whenever some elements of the family $\F$ do not restrict to be pure on $\fB$. The reader should compare with Example \ref{E:weakunpexample}.
%In light of this consequence, it is an easy matter to formulate an analogue of Corollary \ref{C:unphyper} using Theorem \ref{T:weppnorm}. As the statement is somewhat unwieldy, we leave the details to the interested reader. 

\begin{corollary}\label{C:wephyper}
Let $\fS$ be a separable operator system and let $\fA=\rC^*(\fS)$. Assume that every irreducible $*$-representation of $\fA$ is a boundary representation for $\fS$. Let $\pi:\fA\to B(\H)$ be a unital $*$-representa\-tion such that $\pi(\fA)$ has the weak expectation property, and let $\Pi:\fA\to B(\H)$ be a unital completely positive extension of $\pi|_\fS$. Then, $\pi=\Pi$ if and only if there is a family of states on $\rC^*(\Pi(\fA))$ which separate $(\Pi-\pi)(\fA)$ and have the unique extension property with respect to $\pi(\fA)$. 
\end{corollary}
\begin{proof}
Assume that there is a family of states on $\rC^*(\Pi(\fA))$ which separate $(\Pi-\pi)(\fA)$ and have the unique extension property with respect to $\pi(\fA)$. We may apply Theorem \ref{T:weppnorm} to the inclusion $\pi(\fA)\subset \rC^*(\Pi(\fA))$ (see (\ref{Eq:inclusion})) to find a (potentially different) family of states on $\rC^*(\Pi(\fA))$ which separate $(\Pi-\pi)(\fA)$ and restrict to be pure on $\pi(\fA)$. Consequently, $\pi=\Pi$ by virtue of Theorem \ref{T:purerestriction}. The converse is trivial.
\end{proof}

We draw the reader's attention to the main point of Corollary \ref{C:wephyper}: unlike in Theorem \ref{T:purestateuep}, the separating family is not assumed to consist of \emph{pure} states.

We finish by mentioning that it would be of interest to obtain a version of Theorem \ref{T:unpstate} or Corollary \ref{C:unphyper} based on Theorem \ref{T:weakunperexample}. It is not clear to us how this can be achieved at present. The promise of such an application of Theorem \ref{T:weakunperexample} is the reason why we chose to state it in the context of $\fB$ being separable. The reader will notice that this condition can be removed at the cost of obtaining a net rather than a sequence. We opted for the current version as sequences seem more appropriate for arguments relying on integration techniques.

\bibliography{/Users/raphaelclouatre/Dropbox/Research/Bibliography/biblio_main}
%\bibliography{/Users/Raphael/Dropbox/Research/Bibliography/biblio_main}
\bibliographystyle{plain}

\end{document}